\documentclass[12pt]{amsart}
\usepackage{amssymb,bbold}
\usepackage[all]{xy}

\theoremstyle{plain}
\newtheorem{theorem}{Theorem}

\newtheorem{lemma}[theorem]{Lemma}

\theoremstyle{definition}

\begin{document}
\baselineskip 18pt

\title[The Krengel's theorem for compact operators ]
      {The Krengel's theorem for compact operators between locally solid vector lattices}

\author[O.~Zabeti]{Omid Zabeti}


\address[O.~Zabeti]
  {Department of Mathematics, Faculty of Mathematics, Statistics, and Computer science,
   University of Sistan and Baluchestan, Zahedan,
   P.O. Box 98135-674. Iran}
\email{o.zabeti@gmail.com}


\begin{abstract}
Suppose $X$ is a locally solid vector lattice. It is known that there are several non-equivalent notions for compact operators on $X$. Furthermore, notion of the $AM$-property in $X$ as an extension for the $AM$-spaces in Banach lattices has been considered, recently. In this paper, we establish a variant of the known Krengel's theorem for different types of compact operators between locally solid vector lattices.

\end{abstract}

\date{\today}

\maketitle
\section{motivation and introduction}
Let us start with some motivation. Let $E$ be a Banach lattice. $E$ is called an $AM$-space provided that for each positive $x,y\in E_{+}$, we have $\|x\vee y\|=\|x\|\vee \|y\|$. The remarkable Kakutani's theorem states that every $AM$-space is a closed sublattice of $C(K)$ for some compact Hausdorff space $K$. Now, suppose $E$ is a Banach lattice and $F$ is an $AM$-space. The Krengel's theorem states that every compact operator $T:E\to F$ has a modulus which is defined by the Riesz-Kantorovich formulae; that is $|T|(x)=\sup\{|Ty|: |y|\leq x\}$ for each $x\in E_{+}$. So, we conclude that $AM$-spaces have many interesting properties among the category of all Banach lattices. Therefore, it is fascinating and significant to consider the $AM$-spaces and numerous applications in the operator theory to the locally solid vector lattices and operators between them. The first step has been done in \cite{Z1}; namely, the $AM$-property which is the right extension for the $AM$-spaces. Moreover, Observe that there are several different ways to define bounded and compact operators between locally solid vector lattices. Some applications of the $AM$-property in these classes of operators have been obtained in \cite{EGZ,Z1}.
In this paper, we are going to generalize the known Krengel's theorem \cite[Theorem 5.7]{AB} for different types of compact operators between locally solid vector lattices.

For undefined terminology and related notions, see \cite{AB1,AB}. All locally solid vector lattices in this note are assumed to be Hausdorff.

\section{main result}
First, we recall the notion of the $AM$-property; for more details, see \cite{Z1}.
Suppose $X$ is a locally solid vector lattice. We say that $X$ has the  $AM$-property provided that for every bounded set $B\subseteq X$, $B^{\vee}$ is also bounded with the same scalars; namely, given a zero neighborhood $V$ and any positive scalar $\alpha$ with $B\subseteq \alpha V$, we have  $B^{\vee}\subseteq \alpha V$. Note that by $B^{\vee}$, we mean the set of all finite suprema of elements of $B$. 

In this part, we recall the following useful fact; for more details, see \cite[Lemma 3]{Z2}.
\begin{lemma}\label{5001}
Suppose $X$ is a locally solid vector lattice with the $AM$-property and $U$ is an arbitrary solid zero neighborhood in $X$. Then, for each $m\in \Bbb N$, $U\vee\ldots\vee U=U$, in which $U$ is appeared $m$-times.
\end{lemma}
Moreover, we have the following useful inequality in Archimedean vector lattices.
\begin{lemma}\label{7001}
suppose $E$ is an Archimedean vector lattice. Then for $x_1,\ldots,x_n$ and $y_1,\ldots,y_n$ in $E$, the following inequality holds.
\[x_1\vee\ldots\vee x_n-y_1\vee\ldots\vee y_n\leq (x_1-y_1)\vee\ldots\vee (x_n-y_n).\]
\end{lemma}
\begin{proof}
We proceed the proof by induction. For $n=2$, we have
\[x_1\vee x_2-y_1\vee y_2=(x_1-(y_1\vee y_2))\vee (x_2-(y_1\vee y_2))=(x_1+((-y_1)\wedge (-y_2)))\vee(x_2+((-y_1)\wedge (-y_2)))=
\]
\[((x_1-y_1)\wedge (x_1-y_2))\vee((x_2-y_1)\wedge(x_2-y_2))\leq (x_1-y_1)\vee(x_2-y_2).\]
Now, suppose for $n=k$, the statement is valid. We need prove it for $n=k+1$. By using validness of the result for $n=2$ and $n=k$, we have
\[x_1\vee\ldots\vee x_k\vee x_{k+1}-y_1\vee\ldots\vee y_k\vee y_{k+1}\leq ((x_1\vee\ldots\vee x_k)-(y_1\vee\ldots\vee y_k))\vee (x_{k+1}-y_{k+1})\]
\[\leq (x_1-y_1)\vee\ldots\vee(x_k-y_k)\vee(x_{k+1}-y_{k+1}).\]
\end{proof}
Recall that a subset $B$ of a topological vector space $X$ is said to be totally bounded if for each arbitrary zero neighborhood $V\subseteq X$ there is a finite set $F$ such that $B\subseteq F+V$; for more information, see \cite{AB}.
\begin{lemma}\label{4001}
Suppose $X$ is a locally solid vector lattice with the $AM$-property. If $B\subseteq X$ is totally bounded, then so is $B^{\vee}$. In particular, $\sup B$ exists in $X$ and $\sup B\in \overline{B^{\vee}}$.
\end{lemma}
\begin{proof}
Choose arbitrary solid zero neighborhood $U\subseteq X$. By the assumption, there exists a finite set $F\subseteq X$ such that $B\subseteq F+U$. Assume that $F=\{z_1,\ldots,z_m\}$. Put $z_0=z_1\vee\ldots\vee z_m$. We claim that $B^{\vee}\subseteq \{z_0\}+U$.
Given any $x_1,\ldots, x_n \in B$. There are some $z_1,\ldots, z_n$ ( possibly with the repetition), such that $x_i-z_i\in U$ for all $i=1,\ldots, n$.
Therefore, by using Lemma \ref{7001} and Lemma \ref{5001}, we have
\[x_1\vee\ldots \vee x_n -z_1\vee\ldots\vee z_n\leq (-z_1+x_1)\vee\ldots\vee(-z_n+x_n)\in U\vee\ldots\vee U=U.\]
Since $U$ is solid, similarly, we have
\[z_1\vee\ldots \vee z_n -x_1\vee\ldots\vee x_n\leq (z_1-x_1)\vee\ldots\vee(z_n-x_n)\in U\vee\ldots\vee U=U.\]
This means that $(x_1\vee\ldots\vee x_n)-z_0\in U$.

Now, assume that $D$ is the set of all finite subsets of $B$ directed by the inclusion $\subseteq$. For each $\alpha \in D$, put $g_{\alpha}=\sup \alpha$. Observe that $\{g_{\alpha}\}\subseteq B^{\vee}$ satisfies $g_{\alpha}\uparrow$. By compactness of $\overline{B^{\vee}}$, there exists a subnet of $(g_{\alpha})$ that converges to some $g\in \overline{B^{\vee}}$. Therefore, $\sup B=\sup B^{\vee}=\sup \{g_{\alpha}\}=g$.
\end{proof}
Now, we are able to consider a version of the Krengle's theorem (\cite[Theorem 5.7]{AB})for each class of compact operators between locally solid vector lattices. First, we recall some preliminaries which are needed in the sequel.

Suppose $X$ and $Y$ are locally solid vector lattices and $T:X\to Y$ is a linear operator. $T$ is called $nb$-bounded if there is a zero neighborhood $U\subseteq X$ such that $T(U)$ is also bounded in $Y$; $T$ is said to be $bb$-bounded if it maps bounded sets into bounded sets.

Moreover, a linear operator $T:X\to Y$ is said to be $nb$-compact provided that there is a zero neighborhood $U\subseteq X$ such that $\overline{T(U)}$ is compact in $Y$; $T$ is $bb$-compact if for every bounded set $B\subseteq X$, $\overline{T(B)}$ is compact in $Y$. It is obvious that every $nb$-compact operator is $nb$-bounded and every $bb$-compact operator in $bb$-bounded. These classes of operators enjoy some topological and lattice structures; for a detailed exposition as well as related notions about bounded and compact operators see \cite{EGZ,Tr1,Z1}.

Krengel has proved that when the range of a compact operator $T$ between Banach lattices is an $AM$-space, then the modulus of $T$ exists and is also compact (see \cite[Theroem 5.7]{AB}). In the following, we prove this remarkable result for $nb$-compact operators as well as for $bb$-compact operators when the range space has the $AM$-property.
\begin{theorem}
Suppose $X$ and $Y$ are locally solid vector lattices such that $Y$ possesses the $AM$-property and $T:X\to Y$ is an $nb$-compact operator. Then the modulus of $T$ exists and is also $nb$-compact.
\end{theorem}
\begin{proof}
 There exists a zero neighborhood $U\subseteq X$ such that $T(U)$ is totally bounded in $Y$.
 Observe that for each $x\in U_{+}$, $T[-x,x]$ is totally bounded in $Y$ so that by Lemma \ref{4001}, the supremum $|T|(x)=\sup\{|Ty|: |y|\leq x\}=\sup T[-x,x]$ exists in $Y$. Thus, by \cite[Theorem 1.14]{AB}, the modulus of $T$ exists. According to Lemma \ref{4001}, $\overline{{T(U)}^{\vee}}$ is also compact and $|T|(x)\in \overline{{T(U)}^{\vee}}$. Therefore, $|T|(U_{+})\subseteq \overline{{T(U)}^{\vee}}$. Since $U\subseteq U_{+}-U_{+}$, the proof would be complete.
\end{proof}
\begin{theorem}
Suppose $X$ and $Y$ are locally solid vector lattices such that $Y$ possesses the $AM$-property and $T:X\to Y$ is a $bb$-compact operator. Then the modulus of $T$ exists and is also $bb$-compact.
\end{theorem}
\begin{proof}
The proof essentially has the same line. Fix a bounded set $B\subseteq X$ such that $T(B)$ is totally bounded in $Y$; by replacing $B$ with $Sol(B)$, if necessary, we may assume that $B$ is solid.
 Observe that for each $x\in B_{+}$, $T[-x,x]$ is totally bounded in $Y$ so that by Lemma \ref{4001}, the supremum $|T|(x)=\sup\{|Ty|: |y|\leq x\}=\sup T[-x,x]$ exists in $Y$. Thus, by \cite[Theorem 1.14]{AB}, the modulus of $T$ exists. According to Lemma \ref{4001}, $\overline{{T(B)}^{\vee}}$ is also compact and $|T|(x)\in \overline{{T(B)}^{\vee}}$. Therefore, $|T|(B_{+})\subseteq \overline{{T(B)}^{\vee}}$. Since $B\subseteq B_{+}-B_{+}$, we have the desired result.
\end{proof}

\end{document}